\documentclass[11pt]{article}
\usepackage{times}
\usepackage[utf8]{inputenc}
\usepackage{indentfirst}
\usepackage{graphicx,tikz}
\usepackage[T1]{fontenc}

\usepackage{amsthm}
\newtheorem{problem}{Question}
\newtheorem{definition}[problem]{Definition}
\newtheorem{theorem}[problem]{Theorem}
\newtheorem{observation}[problem]{Observation}
\newtheorem{lemma}[problem]{Lemma}
\newtheorem{conjecture}[problem]{Conjecture}
\newtheorem{corollary}[problem]{Corollary}

\title{The Width of a Ball in a Hypercube}
\author{Kada K Williams}
\date{March 15, 2024}

\begin{document}

\maketitle

\section{Introduction}

In \cite{Spe}, Sperner asked how many subsets of $[n]$, the set $\{1,2,\dots,n\}$, can be selected without having one that is a proper subset of another. Since then, this question has been extended to finite partially ordered sets -- posets for short -- asking how many pairwise incomparable elements can be picked. The answer to this question is referred to as the width of a poset.

Certainly, in a chain, a poset whose order is total, at most one element can be chosen. Therefore, the number of pairwise incomparable elements is limited by the number of chains required to cover the poset. Dilworth \cite{Dil} proved that conversely, denoting by $w$ the width of a poset $P$, one can partition $P$ into $w$ many chains. The subset of a poset formed by pairwise incomparable elements is called an antichain.

Expanding on Dilworth's ideas, if $(P,<)$ is a poset and $x\in P$, then we consider the height of $x$, the largest non-negative value $h$ such that one can find elements $x_0,x_1,\dots,x_h\in P$ such that $x_0<x_1<\dots<x_h$ and $x=x_h$. For example, regarding subsets of $[n]$, any $k$-element subset has height $k$, because descending to a proper subset means removing at least one element. Let $P_h$, the layer $h$, denote the set of $x\in P$ with height $h$. Then for $P=\mathcal{P}([n])$, the family of subsets of $[n]$, we observe that $P_h$ is the family of $h$-element subsets.

A modern approach to Sperner's theorem \cite{LYM} exploits the symmetry of $\mathcal{P}([n])$. First, one sees that $\emptyset \subset [1]\subset [2]\subset \dots \subset [n]$ is a maximal chain that contains an element in each layer. However, upon permuting $\{1,2,\dots,n\}$, the set $[k]$ maps to every point in layer $k$ equally. Now if $A$ is an antichain, then at most one $x\in A$ can appear in a chain. In particular, if $x$ is in layer $k$, then its chance of occurring is ${n \choose k}^{-1}$. The sum of these chances over all $x\in A$ cannot exceed $1$:
$$\sum_{x\in A} {n \choose h(x)}^{-1}\le 1.$$

More generally, let $P$ be a poset where any antichain $A$ satisfies
$$\sum_{x\in A} \Big|P_{h(x)}\Big|^{-1}\le 1.$$
Such a poset is described as a KLYM poset \cite{DFr} and has the simple property that its width equals the size of its largest layer. In this vein, the width of $\mathcal{P}([n])$ is clearly ${n \choose \lfloor n/2\rfloor }$.

Nevertheless, Greene and Kleitman \cite{GrK} obtained an explicit construction that partitions $P=\mathcal{P}([n])$ into $w(P)$ many chains. Actually, their bracketing strategy extends to the case of a multiset power set $P$, where $i\in [n]$ can be contained not only $0$ or $1$ times, but up to $\mu_i$ times. The poset in question is then isomorphic to the set of positive divisors of $N=\prod_{i=1}^{n}p_i^{\mu_i}$ under the relation of proper divisors. Wang and Yeh \cite{WaY} provide a novel proof that $P$ has the KLYM property, whilst $\frac{|P_h|}{|P_{h-1}|}$ is decreasing in $h$, via product argument.

The problem of calculating or bounding $w(P)$ for an arbitrary subposet $P$ of $\mathcal{P}([n])$, such as a downset \cite{DHL}, or the intersection of a downset and an upset, continues to be wide open. In the spirit of Richard Hamming's recurring question of which problem in this field is most important, we restrict our focus to the subposet obtained from a $p$-element set by removing or adding at most $r$ elements.

\section{Sublayers}

Write $[n]=\{1,\dots,p\}\cup \{p+1,\dots,p+q\}$, where $q=n-p$. For the family of sets obtained by subtracting $i$ elements of $\{1,\dots,p\}$ and adding $j$ elements from $\{p+1,\dots,p+q\}$, write $X_{i,j}$. Here, the subtraction and addition are independent, in the sense that $|X_{i,j}|={p\choose i}\cdot {q\choose j}$. Since the sets in $X_{i,j}$ have $p-i+j$ elements, their height shall depend on the value $i-j$.

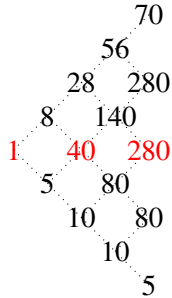
\begin{figure}[ht]
\centering
\begin{tikzpicture}[scale=0.45]
\draw[dotted] (0,0) -- (4,4);
\draw[dotted] (1,-1) -- (4,2);
\draw[dotted] (2,-2) -- (4,0);
\draw[dotted] (3,-3) -- (4,-2);
\draw[dotted] (0,0) -- (4,-4);
\draw[dotted] (1,1) -- (4,-2);
\draw[dotted] (2,2) -- (4,0);
\draw[dotted] (3,3) -- (4,2);

\node[red] at (0,0) {1};
\node at (1,-1) {5};
\node at (1,1) {8};
\node at (2,-2) {10};
\node[red] at (2,0) {40};
\node at (2,2) {28};
\node at (3,-3) {10};
\node at (3,-1) {80};
\node at (3,1) {140};
\node at (3,3) {56};
\node at (4,-4) {5};
\node at (4,-2) {80};
\node[red] at (4,0) {280};
\node at (4,2) {280};
\node at (4,4) {70};

\end{tikzpicture}
\caption{The sublayers of $B_4[5,8]$ and its largest layer}
\label{subs}
\end{figure}

\begin{definition}
    The ball of radius $r$ centered at $[p]$ in $\mathcal{P}([p+q])$ is given by
    $$B_r[p,q]=\bigcup_{i+j\le r} X_{i,j}.$$
\end{definition}

\begin{conjecture}
    The largest layer of $B_r[p,q]$ is its unique largest antichain.
\end{conjecture}

We can prove this conjecture in most cases, because $B_r[p,q]$ has the symmetries of relabelling $\{1,\dots,p\}$ and $\{p+1,\dots,p+q\}$, and the $|X_{i,j}|$ are well-behaved.

\begin{lemma} 
    In the poset $B_r[p,q]$, let $\mathcal{C}$ be a weighted collection of chains that contains a set in every layer. Of those, let $\mathcal{C}_{i,j}$ contain the chains that pass through $X_{i,j}$. In view of the proportions $\frac{|\mathcal{C}_{i,j}|}{|X_{i,j}|}$, if there is a layer such that these proportions are no larger outside than within, that layer is a maximal antichain.
\end{lemma}

\begin{proof}
    Without loss of generality, $\mathcal{C}$ passes through all the sets in $X_{i,j}$ an equal number of times. Indeed, $X_{i,j}$ is an orbit of the action that relabels $\{1,\dots,p\}$ and $\{p+1,\dots,p+q\}$. Furthermore, if $\mathcal{C}'$ includes all $p!\cdot q!$ many images of each $C\in \mathcal{C}$, then the proportions are scaled as $\frac{|\mathcal{C}'_{i,j}|}{|X_{i,j}|}=p!\cdot q!\cdot \frac{|\mathcal{C}_{i,j}|}{|X_{i,j}|}$, thus yielding the number of $C\in \mathcal{C}'$ containing any given $x\in X_{i,j}$. Therefore, we may assume that $\mathcal{C}(x)=\{C\in\mathcal{C}|x\in C\}$ has no larger size outside layer $L$ than within.

    Now for any antichain $A$, a chain from $\mathcal{C}$ can contain at most one $x\in A$. Hence, the $\mathcal{C}(x)$ are disjoint, implying 
    $$\sum_{x\in A}|\mathcal{C}(x)|\le |\mathcal{C}|.$$
    Since each $C\in \mathcal{C}$ visits every layer, $A=L$ is a case of equality. Given that $|A|=w=w(P)\ge |L|$, the left-hand sum is minimal when it contains the smallest possible values. If $w>|L|$, then these values are given by the values across $x\in L$, and more, in contradiction with how the sum over $x\in L$ equals $1$. It follows that $L$ is a maximal antichain. Moreover, if $w=|L|$, and values outside $L$ are strictly larger than values within, then only $A=L$ is maximal.
\end{proof}

Our proof strategy, then, involves finding the largest layer and selecting a collection of chains that passes through every point of this layer once, but includes a point in $X_{i,j}$ no less than $|X_{i,j}|$ times for all $i,j$.

\begin{observation}
    For fixed radius $i+j-1$, the ratios $\frac{|X_{i-1,j}|}{|X_{i,j-1}|}$ are decreasing in $j$.
\end{observation}

\begin{proof}
    Since ${p\choose i}{p\choose {i-1}}^{-1}=\frac{p-i+1}{i}$, it is clear that 
    $$\frac{|X_{i-1,j}|}{|X_{i,j-1}|}=\frac{(q-j+1)i}{(p-i+1)j}$$
    If $j\nearrow$, then $i\searrow$, $(q-j+1)i\searrow$, $(p-i+1)j\nearrow$, so the fraction decreases.
\end{proof}

\begin{corollary}
    The spheres $S_r[p,q]=\bigcup_{i+j=r} X_{i,j}$ have the KLYM property.
\end{corollary}

\begin{proof}
    Let $\mathcal{C}$ contain any chain that proceeds from a minimal point by successively adding an element of $\{p+1,\dots,p+q\}$ and removing an element of $\{1,\dots,p\}$. Upon relabelling those elements, we obtain a list of chains that passes through every point in a layer equally. We readily deduce the KLYM inequality.
\end{proof}

Analytically, $\frac{|X_{i-1,j}|}{|X_{i,j-1}|}=1$ holds exactly when $\frac{(q-j+1)i}{(p-i+1)j}=1$, which rearranges to a linear condition $\frac{i}{j}=\frac{p+1}{q+1}$. This determines the largest layer of $S_r[p,q]$.

Polynomially, $|X_{i,j}|$ has degree $i+j$ in $(p,q)$, indicating that $B_r[p,q]$ resembles $S_r[p,q]$ in general. On the other hand, there exist $r,p,q$ for which the largest layer of $B_r[p,q]$ is different from the largest layer of $S_r[p,q]$.

\begin{figure}[ht]
\centering
\begin{tikzpicture}[scale=0.5]
\draw[dotted] (0,0) -- (10,10);
\draw[dotted] (1,-1) -- (10,8);
\draw[dotted] (2,-2) -- (10,6);
\draw[dotted] (3,-3) -- (10,4);
\draw[dotted] (4,-4) -- (10,2);
\draw[dotted] (5,-5) -- (10,0);
\draw[dotted] (6,-6) -- (10,-2);
\draw[dotted] (7,-7) -- (10,-4);
\draw[dotted] (8,-8) -- (10,-6);
\draw[dotted] (9,-9) -- (10,-8);
\draw[dotted] (0,0) -- (10,-10);
\draw[dotted] (1,1) -- (10,-8);
\draw[dotted] (2,2) -- (10,-6);
\draw[dotted] (3,3) -- (10,-4);
\draw[dotted] (4,4) -- (10,-2);
\draw[dotted] (5,5) -- (10,0);
\draw[dotted] (6,6) -- (10,2);
\draw[dotted] (7,7) -- (10,4);
\draw[dotted] (8,8) -- (10,6);
\draw[dotted] (9,9) -- (10,8);

\node at (0,0) {1};
\node at (1,-1) {9};
\node at (1,1) {17};
\node at (2,-2) {36};
\node at (2,0) {153};
\node at (2,2) {136};
\node[yshift=.5em] at (10,2) {1559376};
\node[yshift=.5em] at (10,4) {1633632};
\node[yshift=-0.5em] at (8,2) {519792};
\node[yshift=-0.5em] at (8,4) {445536};

\draw[orange] (4,4) -- (10,4);
\draw[purple] (2,2) -- (10,2);
\draw[dashed] (-1,1) -- (13,5);
\end{tikzpicture}
\caption{In $B_{10}[9,17]$, sphere layers decrease downwards from below the divide}
\label{subs}
\end{figure}
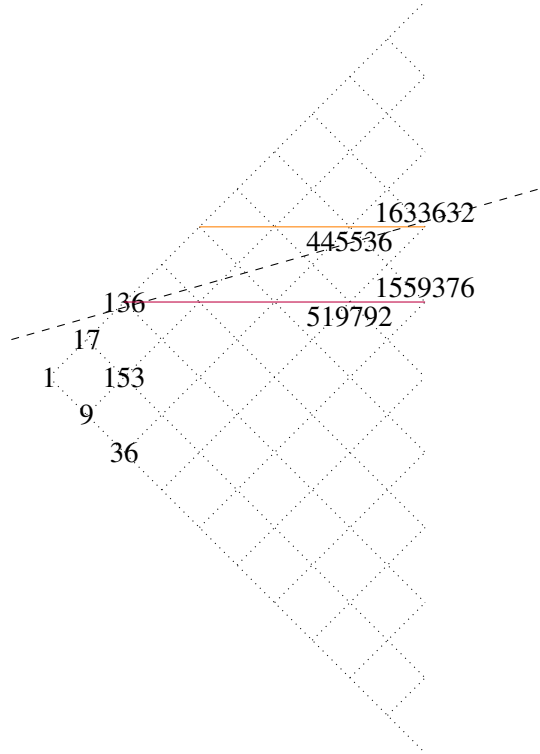

Notice that ${p\choose i}$ increases while $i\le \frac12p$, and the same is true of ${q\choose j}$, showing that provided $i+j<\frac12(p+q)$, incrementing $i$ or $j$ yields an $X_{i,j}$ at least as large. Hence, if $r\le \frac12(p+q)$, then the largest sublayer of $B_r[p,q]$ is on $S_r[p,q]$. For this $X_{i,j}$, let us take one chain through each $x\in X_{i,j}$ and anticipate where it continues.

\section{Conclusion}

\begin{theorem}
    Let $r\le \frac12(p+q)$. Then abbreviating $B_r=B_r[p,q]$ and $S_r=S_r[p,q]$,
    $$w(B_r)\le w(S_{2r})+w(S_{2r-2})+w(S_{2r-4})+\dots$$
    In fact, if $\max(p,q)=\Omega(r^3)$, then Conjecture 2 holds.
\end{theorem}

\begin{proof}
    As we have seen, $\frac{i+1}{j}>\frac{p+1}{q+1}$ if and only if $|X_{i,j}|>|X_{i+1,j-1}|$. The solution of $\frac{i+1}{j}=\frac{p+1}{q+1}$, given $i+j=r$, satisfies $i+1+j\le \frac12(p+1+q+1)$, whence $i+1\le \frac12(p+1)$ and $j\le \frac12(q+1)$. Rounding $j$ down and $i$ up, we obtain the indices of a largest layer $X_{i,j}\subset S_r$. It follows that $S_{r-1}$ has largest layer either $X_{i-1,j}$ or $X_{i,j-1}$, both of which have size not exceeding that of $X_{i,j}$.

    If $\mathcal{C}$ contains one chain that passes through every layer of $S_r\cup S_{r-1}$, we are able to deduce that its width is at most $|X_{i,j}|$. The inequality follows by induction.

    Interchanging $p$ and $q$, let $q=\max(p,q)$. Further, consider that 
    $$|X_{i,j}|-|X_{i+1,j-1}|\ge |X_{i,j-2}|$$
    or equivalently, $\frac{q-j+1}{j}-\frac{p-i}{i+1}\ge \frac{j-1}{q-j+2}$. For the values of $i,j$ as before, the left-hand fraction is at least $\frac1{j(i+1)}$, from least common denominators. As $q-j+2\ge q-r+3$, we now have $q-r+3\ge j(j-1)(i+1)$, implied by $q\ge \left(\frac{r+\frac12}{3}\right)^3+r-3=\Omega(r^3)$.

    Wishing for a list $\mathcal{C}$ of chains in $B_r$ that strike every $X_{i,j}$ at least once per point, let us begin with striking the largest $X_{i,j}\subset S_r$ as above. Directing all chains through $X_{i,j-1}$, it has been struck, and then we can strike $X_{i,j-2}$ as well as $X_{i+1,j-1}$, according to the bound. This can be iterated all the way down, as the bound stays in range. Going upward from $X_{i,j}$, it suffices to zigzag in $S_{r-1}\cup S_r$. Since chains through that $X_{i,j}$ can fulfill all sublayers it covers, we are free to keep $x\cap [p]$ fixed for chains descending from $x\in X_{i-l,j-l}$, $l=1,2,\dots$, zigzagging upwards. This completes the proof, according to the Lemma.
\end{proof}

Although our Conjecture can be verified case by case with ideas presented here, a comprehensive proof is unknown as of yet.

\textsc{Department of Pure Mathematics and Mathematical Statistics, University \
of Cambridge, Wilberforce Road, Cambridge CB3 0WB.} \\ \\
\textit{E-mail address:} \texttt{kkw25@cam.ac.uk}

\end{document}